\documentclass[10pt]{amsart}

\usepackage{amsfonts,amssymb}

\input{xy}
\xyoption{all}

\newtheorem{proposition}{Proposition}[section]
\newtheorem{lemma}[proposition]{Lemma}
\newtheorem{corollary}[proposition]{Corollary}
\newtheorem{theorem}[proposition]{Theorem}
\theoremstyle{definition}

\newtheorem{example}[proposition]{Example}

\theoremstyle{remark}

\newcommand{\thlabel}[1]{\label{th:#1}}
\newcommand{\thref}[1]{Theorem~\ref{th:#1}}
\newcommand{\selabel}[1]{\label{se:#1}}
\newcommand{\seref}[1]{Section~\ref{se:#1}}

\newcommand{\colabel}[1]{\label{co:#1}}
\newcommand{\coref}[1]{Corollary~\ref{co:#1}}

\newcommand{\exlabel}[1]{\label{ex:#1}}
\newcommand{\exref}[1]{Example~\ref{ex:#1}}

\newcommand{\eqlabel}[1]{\label{eq:#1}}
\newcommand{\equref}[1]{(\ref{eq:#1})}

\def\ra{\rightarrow}

\def\Id{{\rm Id}}

\newcommand\smi{\mbox{$S^{-1}$}}

\def\ot{\otimes}
\def\va{\varepsilon}

\def\un{\underline}

\def\mf{\mathfrak}

\def\l{\lambda}

\def\va{\varepsilon}
\def\v{\varphi}

\def\rh{\rightharpoonup}
\def\lh{\leftharpoonup}

\def\ra{\rightarrow}
\def\a{\alpha}
\def\b{\beta}

\def\ov{\overline}
\def\cal{\mathcal}
\def\un{\underline}

\def\equal#1{\smash{\mathop{=}\limits^{#1}}}

\allowdisplaybreaks[4]

\begin{document}
\title[Ribbon quasi-Hopf algebras]
{Some ribbon elements for the quasi-Hopf algebra $D^\omega(H)$}
\author{Daniel Bulacu}
\address{Faculty of Mathematics and Computer Science, University
of Bucharest, Str. Academiei 14, RO-010014 Bucharest 1, Romania}
\email{daniel.bulacu@fmi.unibuc.ro}
\author{Florin Panaite}
\address{Institute of Mathematics of the Romanian Academy, PO-Box 1-764, RO-014700 Bucharest, Romania}
\email{Florin.Panaite@imar.ro}
\subjclass[2010]{16T05; 18D10}

\begin{abstract}  
We construct an explicit isomorphism between the quasitriangular quasi-Hopf algebra $D^\omega(H)$ defined in \cite{bp} 
and a certain quantum double quasi-Hopf algebra. We give also new characterizations for a quasitriangular quasi-Hopf algebra 
to be ribbon and use them to construct some ribbon elements for $D^\omega(H)$.  
\end{abstract}

\keywords{Quasi-Hopf algebra; twisted quantum double; ribbon element.}
\maketitle

\section{Introduction}
Dijkgraaf, Pasquier and Roche constructed in \cite{dpr}, from a finite group $G$ and a normalized $3$-cocycle $\omega$ on it, 
the famous quasi-Hopf algebra $D^\omega(G)$ (called the twisted quantum double of $G$). 
The importance of this construction stems from the fact that 
the irreducible representations of $D^\omega(G)$ allow to recover the fusion rules, the $S$ matrix and 
the conformal weights of a certain Rational Conformal Field Theory described in \cite{DVVV}. 
On the other hand, ribbon (quasi-) quantum groups give rise to topological invariant of knots and links: following 
the constructions of Reshetikhin and Turaev \cite{rt1, rt2}, for such an algebra one can define regular isotopy invariants of 
coloured ribbon graphs, the colours being finite-dimensional representations. This result was applied to the ribbon quasi-Hopf algebra 
$D^\omega(G)$ in \cite{ac} by considering surgery on the ribbon graphs coloured by a representation of $D^\omega(G)$, 
leading thus to a $3$-manifold invariant.  
 
In \cite{bp} we generalized the construction of $D^\omega(G)$ to an arbitrary finite-dimensional cocommutative Hopf algebra 
$H$ and $\omega : H\otimes H\otimes H\rightarrow k$ a normalized 3-cocycle in the cohomology of commutative
algebras over cocommutative Hopf algebras introduced by Sweedler in \cite{sweed}. We denoted this new quasi-Hopf algebra by 
$D^\omega(H)$ and showed that $D^\omega(H)$ is always a quasitriangular (QT for short) quasi-Hopf algebra. We expected $D^\omega(H)$ to always be ribbon, as 
$D^\omega(G)$ is, but we were able to provide a ribbon element for it only if an extra condition is satisfied; 
see \cite[Proposition 3.3]{bp}. The goal of this paper is to overcome this problem and to explain the meaning of the extra condition 
required in \cite{bp}. Towards this end, we identify $D^\omega(H)$ with the quantum double (in the sense of Hausser 
and Nill \cite{hndcp, hn2}) of $H^*_\omega$, the finite-dimensional quasi-Hopf algebra introduced in \cite{pvo}; see 
\thref{thequasiHopfalgebraDOmegaH}. This gives us for free 
the QT quasi-Hopf algebra structure of $D^\omega(H)$ as well as the modular elements of $D^\omega(H)$ in terms of the modular elements 
of $H$. The latter occur in the computation of the ribbon elements for $D^\omega(H)$, owing to 
\thref{charactribbonqHa} and its \coref{impribbquantumdouble} below, and the fact that $D^\omega(H)$ is unimodular. It then  
comes out (see \thref{someribbonDomega}) that to any square root $\zeta$ in $G(H^*)$ of the modular element $\mu_H$ we can associate a ribbon element for $D^\omega(H)$, where $G(H^*)$ stands for the group of grouplike elements of $H^*$, that is algebra maps 
from $H$ to $k$. Situations when such an element $\zeta$ exists are uncovered at the end of \seref{ribbDrinDoubqHA},  
for instance when $H$ has odd dimension or $G(H^*)$ is of odd order, or when 
the characteristic of $k$ does not divide the dimension of $H$ (and consequently when $k$ has characteristic zero).             
\exref{canribbelemDoH} says that the extra condition in \cite[Proposition 3.3]{bp} that guarantees a ribbon element for 
$D^\omega(H)$ is satisfied if and only if $H$ is unimodular. As the Hopf group algebra $k[G]$ is always unimodular, this explains why for the quasi-Hopf algebra $D^\omega(G)$ one can always construct a ribbon element. We should also mention here that our approach provides new ribbon elements even for some of the quasi-Hopf algebras $D^\omega(G)$, and thus new $3$-manifold invariants as in \cite{ac}.            
\section{Preliminaries}\selabel{prelim}
\setcounter{equation}{0}
We present, briefly, the definition an the basic properties of a (quasitriangular) quasi-Hopf algebra. For more information 
we refer to \cite{d2} or \cite{kas}, \cite{maj}.  
We work over a field $k$. All algebras, linear
spaces, etc. will be over $k$; unadorned $\ot $ means $\ot_k$. 

A quasi-bialgebra is
a $4$-tuple $(H, \Delta , \va , \Phi )$, where $H$ is
an associative algebra with unit $1_H$, $\Phi$ is an invertible element in $H\ot H\ot H$, and
$\Delta :\ H\ra H\ot H$ and $\va :\ H\ra k$ are algebra
homomorphisms such that $\Delta$ is coassociative up to conjugation by $\Phi$ and 
$\va$ is counit for $\Delta$; furthermore, $\Phi$ is a normalized $3$-cocycle.  

In what follows we denote $\Delta(h)=h_1\ot h_2$, for all $h\in H$, 
the tensor components of $\Phi$ by capital letters and the ones of $\Phi^{-1}$ by lower case letters. 

$H$ is called a quasi-Hopf
algebra if, moreover, there exists an
anti-morphism $S$ of the algebra
$H$ and elements $\a , \b \in
H$ such that, for all $h\in H$, we
have:
\begin{eqnarray}
&&
S(h_1)\a h_2=\va(h)\a
~~{\rm and}~~
h_1\b S(h_2)=\va (h)\b,\eqlabel{q5}\\ 
&&X^1\b S(X^2)\a X^3=1_H
~~{\rm and}~~
S(x^1)\a x^2\b S(x^3)=1_H.\eqlabel{q6}
\end{eqnarray}

A quasi-Hopf algebra with $\Phi=1_H\ot 1_H\ot 1_H$ and $\a=\b=1_H$ is an ordinary Hopf algebra. 

For a quasi-Hopf algebra $H$ we introduce the following elements in $H\ot H$:
\begin{equation}\eqlabel{pqr}
p_R=p^1\ot p^2:=x^1\ot x^2\b S(x^3)~\mbox{and}~
q_R=q^1\ot q^2:=X^1\ot S^{-1}(\a X^3)X^2.
\end{equation}

Note that our definition of a quasi-Hopf algebra is different from the
one given by Drinfeld \cite{d2} in the sense that we do not 
require the antipode to be bijective. Anyway, the bijectivity 
of the antipode $S$ will be implicitly understood in the case when $S^{-1}$, the inverse 
of $S$, appears is formulas or computations. According to \cite{BCintegrals}, $S$ is always bijective, provided that 
$H$ is finite dimensional. 

The antipode of a Hopf algebra is an 
anti-morphism of coalgebras. For a quasi-Hopf algebra $H$ there is an invertible element 
$f=f^1\ot f^2\in H\ot H$, called the Drinfeld twist, such that $\va(f^1)f^2=
\va(f^2)f^1=1_H$ and 
$f\Delta (S(h))f^{-1}= (S\ot S)(\Delta ^{\rm cop}(h))$, for all $h\in H$, where $\Delta ^{\rm cop}(h)=h_2\ot h_1$. 

For $H$ a quasi-bialgebra, the category of left $H$-modules ${}_H{\cal M}$ is monoidal (see \cite[Chapter XV]{kas} for the definition 
of a monoidal category). The tensor product is defined by $\ot$ 
endowed with the $H$-module structure given by $\Delta$ and unit object $k$, considered as an $H$-module via $\va$; the associativity 
constraint is determined by $\Phi$ and the left and right unit constraints are given by the canonoical isomorphisms in the 
category of $k$-vector spaces.    

A quasi-bialgebra $H$ is called quasitriangular (QT for short) if, moreover, the category ${}_H{\cal M}$ 
is braided in the sense of \cite[Definition XIII.1.1]{kas}. This is equivalent to the existence   
of an invertible element $R=R^1\ot R^2\in H\ot H$ (formal notation, summation 
implicitly understood), called $R$-matrix, satisfying certain conditions. Record that any $R$-matrix obeys  
\begin{equation}\eqlabel{qt4}
\va(R^1)R^2=\va(R^2)R^1=1_H.
\end{equation}

When we refer to a QT quasi-bialgebra or quasi-Hopf algebra we 
always indicate the $R$-matrix $R$ that produces the QT structure by pointing out the couple $(H, R)$. 

For $(H, R)$ a QT quasi-Hopf algebra, $u$ is the element of $H$ defined by    
\begin{equation}\eqlabel{elmu}
u=S(R^2x^2\b S(x^3))\a R^1x^1.
\end{equation}
By \cite{bn3}, $u$ is an invertible element of $H$ and the following equalities hold ($R_{21}:=R^2\ot R^1$):
\begin{eqnarray}\eqlabel{sinau}
&&\hspace*{-1cm}
S^2(h)=uhu^{-1},~\forall~h\in H,\eqlabel{ssinau}\\
&&\hspace*{-1cm}
S(\a )u=S(R^2)\a R^1,\eqlabel{extr}\\
&&\hspace*{-10mm}
\Delta(u)=
f^{-1}(S\ot S)(f_{21})(u\ot u)(R_{21}R)^{-1}.\eqlabel{qrib7}
\end{eqnarray}

\section{Ribbon quasi-Hopf algebras}\selabel{ribbonqHa}
The definition of a ribbon quasi-Hopf algebra is designed in such a way that the category 
${}_H{\cal M}^{\rm fd}$ of finite-dimensional left $H$-modules is a ribbon category as in \cite[Definition XIV.3.2]{kas}. 
More exactly, by \cite{bpvo} a QT quasi-Hopf algebra $(H, R)$ is ribbon if there exists an 
invertible central element $\nu\in H$ such that 
\begin{equation}\eqlabel{ribbonqHa}
\Delta(\nu)=(\nu\ot \nu)(R_{21}R)^{-1}~\mbox{and}~S(\nu)=\nu.
\end{equation}

We can provide the following characterization for ribbon quasi-Hopf algebras. Note that \equref{sribbonqHa1} below was 
proved for the first time in \cite[Proposition 5.5]{bpvo} under the condition $\alpha$ invertible; that it works in general was  
proved in \cite[Section 2.3]{somfact}. In what follows we present an easy proof of this, based mostly on 
the arguments used in \cite{bpvo}.  

\begin{theorem}\thlabel{charactribbonqHa}
A QT quasi-Hopf algebra $(H, R)$ is ribbon if and only if there exists a central element $\nu\in H$ (called ribbon element) 
such that 
the conditions in \equref{ribbonqHa} hold and 
\begin{equation}\eqlabel{sribbonqHa1}
\nu^2=uS(u),
\end{equation}
where, as before, $u$ is the element defined in \equref{elmu}.  
\end{theorem}
\begin{proof}
Suppose that $(H, R)$ is ribbon and let $\nu\in H$ be an invertible central element such that the conditions in \equref{ribbonqHa} are satisfied. To prove that \equref{sribbonqHa1} is satisfied as well, observe first that by applying $\va\ot \va$ to 
both sides of the first equality in \equref{ribbonqHa} we get $\va(\nu)=1$; see \equref{qt4}. Denote $\eta:=\nu^{-1}$ and 
restate the first equality in \equref{ribbonqHa} as $\Delta(\eta)=(\eta\ot \eta)R_{21}R$. Thus, by 
this equality and the fact 
that $\eta$ is a central element in $H$ we deduce that 
\begin{eqnarray*}
&&
\a =\va(\eta)\a 
=S(\eta_1)\a \eta_2
=S(\eta R^2r^1)\a \eta R^1r^2
\equal{\equref{ribbonqHa}}\eta^2S(R^2r^1)\a R^1r^2\\
&&\hspace*{1cm}\equal{\equref{extr}}\eta^2 S(\a r^1)ur^2
\equal{\equref{ssinau}}\eta^2 S(S(r^2)\a r^1)u
\equal{\equref{extr}}\eta^2S(S(\a)u)u
\equal{\equref{ssinau}}\eta^2 S(u)u\a .
\end{eqnarray*}   
This fact allows to compute, for all $A, B\in H$: 
\begin{eqnarray*}
&&\hspace*{-5mm}
A\a B=A\eta^2S(u)u\a B
=\eta^2S(uS^{-1}(A))u\a B\\
&&\hspace*{1cm}\equal{\equref{ssinau}}\eta^2S(u)S^2(A)u\a B
\equal{\equref{ssinau}}\eta^2 S(u)uA\a B.
\end{eqnarray*}
In particular, by taking $A\ot B=S(x^1)\ot x^2\b S(x^3)$, by \equref{q6} we conclude that $1_H=\eta^2S(u)u$, 
and therefore $\nu^2=\eta^{-2}=S(u)u$, as needed. Note that $S(u)u=uS(u)$ is a central element of $H$.

Conversely, let $\nu$ be a central element of $H$ such that the relations in \equref{ribbonqHa} and 
\equref{sribbonqHa1} are fulfilled. Then $\nu$ is invertible because so is $u$, and therefore $\nu$ 
is an invertible central element of $H$. It follows now that $(H, R, \nu)$ is a ribbon quasi-Hopf algebra. 
\end{proof}

It is a well established result that ribbon categories are pivotal (or, equivalently, sovereign) braided categories 
satisfying a certain condition related to some canonical twists; see \cite[Proposition A.4]{henr}. In the case when the category is 
${}_H{\cal M}^{\rm fd}$, with $(H, R, \nu)$ a ribbon quasi-Hopf algebra as in \thref{charactribbonqHa}, 
this result is encoded in \cite[Theorem 3.5]{bp}. 
It asserts that $l\mapsto ul$ defines a one to one correspondence between 
$\{l\in G(H)\mid l^2=u^{-1}S(u)~\mbox{and}~S^2(h)=l^{-1}hl,~\forall~h\in H\}$ and central 
invertible elements $\nu\in H$ 
satisfying \equref{ribbonqHa} and \equref{sribbonqHa1}, i.e. ribbon elements of $H$. Here
\[
G(H)=\{l\in H\mid \mbox{$l$ is invertible with $l^{-1}=S(l)$ and 
$\Delta (l)=(l\otimes l)(S\otimes S)(f_{21}^{-1})f$}\}.
\]

Due to the proof of \cite[Proposition 3.12]{dbbt4}, the conditions $l$ invertible and $l^{-1}=S(l)$ in the definition of an 
element $l\in G(H)$ are redundant, provided that $lS^2(h)=hl$, for all $h\in H$. Otherwise stated, we have the 
following.

\begin{corollary}\colabel{explformrobbonquasiHalg}
Let $(H, R)$ be a QT quasi-Hopf algebra  and $u$ the element defined in \equref{elmu}. 
Then $l\mapsto ul$ defines a one to one correspondence between 
\[
R(H):=\{l\in H \mid l^2=u^{-1}S(u)~,~\Delta(l)=(l\ot l)(S\ot S)(f_{21}^{-1})f~\mbox{\rm and}~lS^2(h)=hl,~\forall~h\in H\}
\]
and ribbon structures $\nu$ on $H$ as in \thref{charactribbonqHa}.
\end{corollary} 

We end this section by pointing out that in the finite-dimensional case the element $u^{-1}S(u)$ can be computed in terms of the 
modular elements $\un{g}\in H$ and $\mu\in H^*$, and the $R$-matrix $R=R^1\ot R^2$ of $H$ 
(we refer to \cite{hn} for the definitions of $\un{g}$ and $\mu$). Namely, by the formula (6.21) in \cite{btfact} 
we have
\begin{equation}\label{fu20}
u^{-1}S(u)=
\mu (X^1R^2p^2S(X^3)f^1)\smi
(S(X^2)f^2)R^1p^1S(\un{g}^{-1}),
\end{equation}   
where $\un{g}^{-1}$ is the inverse of $\un{g}$ in $H$ and $p_R=p^1\ot p^2$ is the element defined in \equref{pqr}. 

\begin{corollary}\colabel{impribbquantumdouble}
If $(H, R)$ is a finite-dimensional unimodular QT quasi-Hopf algebra and $u$ is as in \equref{elmu} then 
$l\mapsto ul$ defines a one to one correspondence between 
\[
\{l\in H \mid l^2=\un{g}~,~\Delta(l)=(l\ot l)(S\ot S)(f_{21}^{-1})f~\mbox{\rm and}~lS^2(h)=hl,~\forall~h\in H\}
\]
and ribbon elements of $H$, where $\un{g}$ is the modular element of $H$. 
\end{corollary}
\begin{proof}
When $H$ is unimodular, i.e. $\mu=\va$, the formula in (\ref{fu20}) gives the equality 
$u^{-1}S(u)=S(\un{g}^{-1})$. As $uS(u)=S(u)u$ and $S^2(u)=u$ 
this implies 
\[
S^2(\un{g}^{-1})=S(u^{-1}S(u))=S^2(u)S(u^{-1})=uS(u)^{-1}=S(u)^{-1}u=(u^{-1}S(u))^{-1}=S(\un{g}).
\]
By using the bijectivity of $S$, we obtain that $S(\un{g}^{-1})=\un{g}$, and so $u^{-1}S(u)=\un{g}$. 
Now everything follows from \coref{explformrobbonquasiHalg}. 
\end{proof}
\section{The quasi-Hopf algebras $D^{\omega }(H)$ and $D^{\omega }(G)$}\selabel{Domega}
\setcounter{equation}{0}

Let $H$ be a cocommutative Hopf algebra with antipode S over a base field $k$. As $H$ is cocommutative, we can introduce a 
simplified version of Sweedler's sigma notation: for $h\in H$, we denote
\begin{eqnarray*}
&&\Delta (h)=h\ot h, \;\;\;\;
(\Id_H\ot \Delta )(\Delta (h))=(\Delta \ot \Id_H)(\Delta (h))=h\ot h\ot h,
\end{eqnarray*}
and so on. With this notation, the antipode and counit axioms read:
\begin{eqnarray*}
&&S(h)h=hS(h)=\varepsilon (h)1_H, \;\;\;\;
\varepsilon (h)h=h\varepsilon (h)=h. 
\end{eqnarray*}

We recall some facts concerning Hopf crossed 
products and cohomology (see \cite{sweed}). 
Let $H$ be a
cocommutative Hopf algebra and $A$ a commutative left $H$-module algebra, with $H$-action denoted by 
$H\otimes A\rightarrow A$, $h\otimes
a\mapsto h\cdot a$. Assume that we are given a linear map 
$\sigma: H\otimes H\rightarrow A$, which is
normalized (that is, $\sigma (1_H, h)=\sigma (h, 1_H)=\varepsilon (h)1_A$
for all $h\in H$) and convolution invertible. Suppose that, moreover,
$\sigma $ satisfies the 2-cocycle condition:
\begin{equation}\eqlabel{2coccond}
\sigma (x, y)\sigma (xy, z)=
[x\cdot \sigma (y, z)]\sigma (x, yz),~ \forall \;x, y, z\in H.
\end{equation}
Then, if we define a multiplication on $A\otimes H$ by
$(a\# h)(b\# g)=a(h\cdot b)\sigma (h, g)\# hg$ 
(for $a\in A$ and $h\in H$ we write $a\# h$ in place of $a\ot h$ in order to distinguish this structure), 
this multiplication is associative and $1_A\otimes 1_H$ is a unit. We denote $A\otimes H$ with this algebra structure 
by $A\# _{\sigma }H$ and called it the Hopf crossed product of $A$ and $H$. 

From now on, for the rest of this section, we assume that $H$ is a finite dimensional 
cocommutative Hopf algebra. Thus, $H^*$ is a commutative Hopf algebra 
with unit $\varepsilon $, counit $\varepsilon (\varphi )=\varphi (1_H)$, multiplication 
$(\varphi \psi )(h)=\varphi (h)\psi (h)$, comultiplication $\Delta (\varphi )=\varphi _1\ot 
\varphi _2$ if and only if $\varphi (hg)=\varphi _1(h)\varphi _2(g)$, and antipode $\ov{S}(\varphi )=\varphi \circ S$, 
where $\varphi, \psi\in H^*$ are arbitrary as well as $h, g\in H$. 

Assume that we are given a $k$-linear map $\omega:
H\otimes H\otimes H\rightarrow k$ that is convolution
invertible and satisfies the conditions: 
\begin{eqnarray}
&&\omega (x, y, zt)
\omega (xy, z, t)=\omega (y, z, t)
\omega (x, yz, t)\omega (x, y, z), \;\;\;
\forall \;x, y, z, t\in H, \\ \label{3.1art}
&&
\omega (1_H, x, y)=\omega (x, 1_H, y)=\omega (x, y, 1_H)=
\varepsilon (x)\varepsilon (y), \;\;\; \forall \;x, y\in H.
\end{eqnarray}
Note that such a map $\omega$ is noting but a normalized 3-cocycle in the Sweedler cohomology of $H$ with 
coefficients in $k$, as defined in \cite{sweed}. 

Since $H$ is finite dimensional, we can identify $(H\ot H\ot H)^*$ with $H^*\ot H^*\ot H^*$, so we 
can regard $\omega \in H^*\ot H^*\ot H^*$; we denote $\omega =\omega _1\ot \omega _2\ot 
\omega _3$ and its convolution inverse $\omega ^{-1}=\ov{\omega}_1\ot \ov{\omega}_2\ot \ov{\omega}_3$. 

We define the element $\Phi \in H^*\ot H^*\ot H^*$ by  $\Phi :=\omega ^{-1}=
\ov{\omega}_1\ot \ov{\omega}_2\ot \ov{\omega}_3$. Since $H^*$ is a commutative algebra, 
$(H^*, \Delta , \varepsilon , \Phi )$ is a quasi-bialgebra, where $\Delta $ and $\varepsilon $ 
are the ones that give the usual coalgebra structure of $H^*$ (dual to the algebra structure 
of $H$). Moreover, if we define $\beta \in H^*$ by the formula $\beta (h)=\omega (h, S(h), h)$, 
then $(H^*, \Delta , \varepsilon , \Phi , \ov{S}, \alpha =\varepsilon , \beta )$ is a 
quasi-Hopf algebra, which will be denoted by $H^*_{\omega }$, see \cite{pvo}.

We can consider the diagonal crossed product $(H^*_{\omega })^*\bowtie H^*_{\omega }$ 
as in \cite{bpv, hndcp}. 
On the other hand, we will construct a certain Hopf crossed product $H^*\# _{\sigma }H$ as follows, see \cite{bp}. 

We introduce first the following notation: $g\triangleleft x=S(x)gx$, for all $g, x\in H$.
Next, we define the linear map $\theta: H\otimes H\otimes H\rightarrow k$, by 
\begin{eqnarray}
&&\theta (g; x, y)=\omega (g, x, y)
\omega (x, y, g\triangleleft (xy))
\omega ^{-1}(x, g\triangleleft x, y), \label{3.2art}
\end{eqnarray}
for all $g, x, y\in H$, where $\omega ^{-1}$ is the convolution inverse
of $\omega $.

It is easy to see that $\theta$ is also 
normalized and convolution invertible. By \cite{bp}, we have 
\begin{eqnarray}
&&\theta (g; x, y)\theta (g; xy,
z)=\theta (g\triangleleft x; y, z)\theta (g; x, yz), \;\;\;\forall \;g, x, y, z\in H.
\label{3.3art}
\end{eqnarray}

Since $H$ is cocommutative, $H^*$ becomes a 
commutative left $H-$module algebra, with action 
$H\otimes H^*\rightarrow H^*$, $h\otimes \varphi \mapsto h\bullet \varphi $,
where $h\bullet \varphi =h\rightharpoonup \varphi \leftharpoonup S(h)$, 
where we denoted by $\rightharpoonup $ and $\leftharpoonup $ the
left and right regular actions of $H$ on $H^*$ given by 
$(h\rightharpoonup \varphi )(a)=\varphi (ah)$ and $(\varphi \leftharpoonup h)(a)=
\varphi (ha)$ for all $h, a\in H$ and $\varphi \in H^*$. Hence,
$(h\bullet \varphi )(a)=\varphi (a\triangleleft h)$ for all $h, a\in H$
and $\varphi \in H^*$.

Define now the linear map $\sigma : H\otimes H\rightarrow H^*$
by $\sigma (x, y)(g)=\theta (g; x, y)$. Since
$\theta $ is normalized and convolution invertible,
$\sigma $ is also normalized and convolution invertible, and 
the relation (\ref{3.3art}) is equivalent to the fact that
$\sigma $ is a 2-cocycle, that is \equref{2coccond} holds if we replace the action $\cdot$ 
by $\bullet$. 
Hence, we can consider the 
Hopf crossed product
$H^*\# _{\sigma }H$, denoted by $D^{\omega }(H)$,
which is an associative algebra with unit 
$\varepsilon \# 1_H$. Its multiplication is given, 
for all $\varphi, \varphi{'}\in H^*$, $h, h{'}\in H$, by
\begin{eqnarray}
&&(\varphi \otimes h)(\varphi {'}\otimes h{'})=\varphi 
(h\rightharpoonup \varphi {'}\leftharpoonup S(h))
\sigma (h,h{'})\otimes hh{'}. \label{3.5art}
\end{eqnarray}

\begin{theorem}\label{izodcpcp}
The linear map 
$w:(H^*_{\omega })^*\bowtie H^*_{\omega }\rightarrow H^*\# _{\sigma }H$ 
defined by 
\begin{eqnarray}
&&\hspace*{-1.5cm}
w(h\bowtie \varphi )=\ov{\omega}_2(h)\ov{\omega}_3(S(h))\ov{\omega}_1 (h\rightharpoonup \varphi 
\leftharpoonup S(h))\# h, \;\;\;\forall \; h\in H, \;\varphi \in H^*, \label{dcpcp}
\end{eqnarray}
is an algebra isomorphism, with inverse $W: H^*\# _{\sigma }H\rightarrow 
(H^*_{\omega })^*\bowtie H^*_{\omega }$ given by 
\begin{eqnarray}
&&\hspace*{-1.5cm}
W(\varphi \# h)=p_1^1(h)p^2(S(h))(\varphi _1\rightharpoonup h
\leftharpoonup S(\varphi _3))\bowtie p_2^1\varphi _2, \;\;\;\forall \;\varphi \in H^*, \;h\in H, 
\label{cpdcp}
\end{eqnarray}
where we denoted by $p^1\ot p^2=x^1\ot x^2\beta S(x^3)$ the element for $H^*_{\omega }$ 
given by \equref{pqr} and by $\rightharpoonup $ and $\leftharpoonup $ the regular 
actions of $H$ on $H^*$ and of $H^*$ on $H$. 
\end{theorem} 
\begin{proof}
We will construct the map $w$ by using the Universal Property of the diagonal crossed product 
(\cite[Proposition 8.2]{ap}). We define the linear maps 
\begin{eqnarray*}
&&\gamma :H^*_{\omega }\rightarrow H^*\# _{\sigma }H, \;\;\;\gamma (\varphi )=\varphi \# 1_H, \\
&&v:H=(H^*_{\omega })^*\rightarrow H^*\# _{\sigma }H, \;\;\;v(h)=\varepsilon \# h. 
\end{eqnarray*}
One can easily see that $\gamma $ is an algebra map and the relations (8.2) and (8.4) in \cite[Proposition 8.2]{ap} 
are satisfied, that is we have 
\begin{eqnarray*}
\gamma (\varphi _1)v(h\leftharpoonup \varphi _2)=v(\varphi _1\rightharpoonup h)\gamma (\varphi _2),~~
v(1_H)=\varepsilon \# 1_H, \;\;\; \forall \; \varphi \in H^*, \;h\in H.
\end{eqnarray*}
So the only thing left to prove is the relation (8.3) in \cite[Proposition 8.2]{ap}, namely
\begin{eqnarray*}
&&\varepsilon \# hh'=(\ov{\ov{\omega}}_1\# 1_H)(\varepsilon \# \omega _1\overline{\omega}_1\rightharpoonup h\leftharpoonup 
\ov{\ov{\omega}}_2)(\omega _2\# 1_H)(\varepsilon \# \overline{\omega}_2\rightharpoonup h'\leftharpoonup 
\ov{\ov{\omega}}_3\omega _3)(\overline{\omega}_3\# 1_H),
\end{eqnarray*}
where we denoted by $\omega ^{-1}=\overline{\ov{\omega}}_1\ot\overline{\ov{\omega}}_2\ot \overline{\ov{\omega}}_3$ 
another copy of $\omega ^{-1}$. We compute:
\begin{eqnarray*}
&&\hspace*{-1.5cm}
(\ov{\ov{\omega}}_1\# 1_H)(\varepsilon \# \omega _1\overline{\omega }_1\rightharpoonup h\leftharpoonup 
\ov{\ov{\omega}}_2)(\omega _2\# 1_H)(\varepsilon \# \overline{\omega}_2\rightharpoonup h'\leftharpoonup 
\ov{\ov{\omega}}_3\omega _3)(\overline{\omega}_3\# 1_H)\\
&=&(\ov{\ov{\omega}}_1\# \omega _1\overline{\omega}_1\rightharpoonup h\leftharpoonup 
\ov{\ov{\omega}}_2)(\omega _2\# \overline{\omega}_2\rightharpoonup h'\leftharpoonup 
\ov{\ov{\omega}}_3\omega _3)(\overline{\omega}_3\# 1_H)\\
&=&\omega _1(h)\overline{\omega}_1(h)\ov{\ov{\omega}}_2(h)\overline{\omega}_2(h')\ov{\ov{\omega}}_3(h')\omega _3(h')
(\ov{\ov{\omega}}_1\# h)(\omega _2(h'\rightharpoonup \overline{\omega}_3\leftharpoonup S(h'))\# h')\\
&=&\omega _1(h)\overline{\omega}_1(h)\ov{\ov{\omega}}_2(h)\overline{\omega}_2(h')\ov{\ov{\omega}}_3(h')\omega _3(h')\\
&&(\ov{\ov{\omega}}_1(h\rightharpoonup \omega _2\leftharpoonup S(h))
(hh'\rightharpoonup \overline{\omega}_3\leftharpoonup S(hh'))\sigma (h, h')\# hh').
\end{eqnarray*}
When we evaluate this in $g\ot \varphi \in H\ot H^*$ we obtain: 
\begin{eqnarray*}
&&\hspace*{-1.5cm}
\omega (h, S(h)gh, h')\omega ^{-1}(h, h', S(hh')ghh')\omega ^{-1}(g, h, h')
\theta (g; h, h')\varphi (hh')\\
&=&\omega (h, S(h)gh, h')\omega ^{-1}(h, h', S(hh')ghh')\omega ^{-1}(g, h, h')\\
&&\omega (g, h, h')\omega (h, h', S(hh')ghh')\omega ^{-1}(h, S(h)gh, h')\varphi (hh')\\
&=&\varepsilon (g)\varphi (hh')
=(\varepsilon \# hh')(g\ot \varphi ), \;\;\; q.e.d.
\end{eqnarray*}
Thus, Proposition 8.2 from \cite{ap} yields an algebra map 
$w:(H^*_{\omega })^*\bowtie H^*_{\omega }\rightarrow H^*\# _{\sigma }H$, defined by 
\begin{eqnarray*}
&&w(h\bowtie \varphi )=\gamma (q^1)v(h\leftharpoonup q^2)\gamma (\varphi ), 
\;\;\; \forall \; \varphi \in H^*, \;h\in H, 
\end{eqnarray*}
where $q^1\ot q^2=X^1\ot S^{-1}(\alpha X^3)X^2$ is the element for $H^*_{\omega }$ given 
by \equref{pqr}. An easy computation shows that this map is identical to the one given by 
(\ref{dcpcp}). 

To prove that $w$ is bijective with inverse $W$, since the underlying vector spaces have the 
same (finite) dimension, it is enough to prove that $w\circ W=\Id $. We compute: 
\begin{eqnarray*}
w(W(\varphi \# h))&=&w(p_1^1(h)p^2(S(h))\varphi _1(h)\varphi _3(S(h))h\bowtie p_2^1\varphi _2)\\
&=&p_1^1(h)p^2(S(h))\varphi _1(h)\varphi _3(S(h))\ov{\omega}_2(h)\ov{\omega}_3(S(h))\\
&&\ov{\omega}_1(h\rightharpoonup p_2^1\varphi _2\leftharpoonup S(h))\# h.
\end{eqnarray*}
When we evaluate this in $g\ot \psi \in H\ot H^*$ we obtain:
\begin{eqnarray*}
&&\hspace*{-1cm}
p_1^1(h)p^2(S(h))\varphi _1(h)\varphi _3(S(h))\ov{\omega}_2(h)\ov{\omega}_3(S(h))
p_2^1(S(h)gh)\varphi _2(S(h)gh)\ov{\omega}_1(g)\psi (h)\\
&=&p^1(gh)p^2(S(h))\varphi (g)\omega ^{-1}(g, h, S(h))\psi (h)\\
&=&\omega _1(gh)\omega _2(S(h))\beta (S(h))\omega _3(h)\omega ^{-1}(g, h, S(h))\varphi (g)\psi (h)\\
&=&\omega (gh, S(h), h)\omega (S(h), h, S(h))\omega ^{-1}(g, h, S(h))\varphi (g)\psi (h).
\end{eqnarray*}
To finish the proof it will be enough to prove that 
\begin{eqnarray*}
&&\omega (gh, S(h), h)\omega (S(h), h, S(h))\omega ^{-1}(g, h, S(h))=\varepsilon (g)\varepsilon (h). 
\end{eqnarray*}
The 3-cocycle condition for $\omega $ applied to the elements $x=h$, $y=S(h)$, 
$z=h$, $t=S(h)$ yields 
\begin{eqnarray}
&&\omega (S(h), h, S(h))=\omega ^{-1}(h, S(h), h). \label{omom-1}
\end{eqnarray}
So it is enough to prove that  
\begin{eqnarray*}
&&\omega (gh, S(h), h)\omega ^{-1}(h, S(h), h)\omega ^{-1}(g, h, S(h))=\varepsilon (g)\varepsilon (h). 
\end{eqnarray*}
But this relation follows immediately by applying the 3-cocycle condition for $\omega $ to the 
elements $x=g$, $y=h$, 
$z=S(h)$, $t=h$. 
\end{proof}

The quantum double $D(H^*_{\omega })$  
of the quasi-Hopf algebra $H^*_{\omega }$ has as underlying algebra structure the diagonal crossed 
product $(H^*_{\omega })^*\bowtie H^*_{\omega }$, so Theorem \ref{izodcpcp} implies: 

\begin{theorem}\thlabel{thequasiHopfalgebraDOmegaH}
$D^{\omega }(H)=H^*\# _{\sigma }H$ is a QT quasi-Hopf algebra. 
\end{theorem}

It turns out that the QT quasi-Hopf algebra structure otained on $D^{\omega }(H)$ 
by transferring the structure from $D(H^*_{\omega })$ via the isomorphism (\ref{dcpcp}) coincides with the one introduced 
in \cite{bp}. Namely: 

$\bullet$ the reassociator: $\Phi =(\ov{\omega}_1\# 1_H)\ot (\ov{\omega}_2\# 1_H)\ot (\ov{\omega}_3\# 1_H)\in D^{\omega }(H)\ot D^{\omega }(H)\ot D^{\omega }(H)$.

$\bullet$ the comultiplication: define the linear map $\gamma :H\ot H\ot H\rightarrow k$ by 
\[
\gamma (g, h; x)=\omega (g, h, x)
\omega (x, g\triangleleft x, h\triangleleft x)\omega ^{-1}(g, x,
h\triangleleft x), 
\]
for all $g, h, x\in H$. Then define the linear map $\nu: H\rightarrow (H\otimes H)^*$,
$\nu (h)(x\otimes y)=\gamma (x, y; h)$. 
Identifying $(H\otimes H)^*$ with $H^*\otimes H^*$, we will write,
for any $h\in H$, $\nu (h)=\nu _1(h)\otimes \nu _2(h)\in
H^*\otimes H^*$. Then the comultiplication of $D^{\omega }(H)$ is defined, 
for all $\varphi \in H^*$, $h\in H$, by 
\begin{eqnarray}
&&\hspace*{-1.2cm}
\Delta : D^{\omega }(H)\rightarrow D^{\omega}(H)\otimes
D^{\omega }(H), \;\;\;
\Delta (\varphi \# h)=
(\nu _1(h)\varphi _1\# h)\otimes
(\nu _2(h)\varphi _2\# h). \label{3.8art}
\end{eqnarray}

$\bullet$ the counit: $\varepsilon :D^{\omega }(H)\rightarrow k$, $\varepsilon (\varphi \# h)=\varphi (1_H)\varepsilon (h)$, for all 
$\varphi \in H^*$, $h\in H$.

$\bullet $ the antipode: $\alpha _{D^{\omega }(H)}=\varepsilon \# 1_H$, 
$\beta _{D^{\omega }(H)}=\beta \# 1_H$ and $s: D^{\omega }(H)\rightarrow D^{\omega }(H)$ given by 
\[
s(\varphi \# h)=
[\varepsilon \# S(h)][\sigma ^{-1}(h, S(h))
S(\varphi \nu _1^{-1}(h))\nu _2^{-1}(h)\# 1_H], 
\]
for all $\varphi \in H^*$, $h\in H$, 
where we denoted by $\nu ^{-1}$ the convolution inverse
of $\nu $, with notation $\nu ^{-1}(h)=
\nu _1^{-1}(h)\otimes \nu _2^{-1}(h)\in H^*\otimes H^*$. 

$\bullet$ the $R$-matrix:
\[
R=\sum _{i=1}^n(e^i\# 1_H)\otimes (\varepsilon
\# e_i)\in D^{\omega }(H)\otimes D^{\omega }(H),
\]
where $\{e_i, e^i\}_i$ are dual bases in $H$ and $H^*$.

Moreover, by \cite{bp}, $\beta$ is convolution invertible with inverse $\beta ^{-1}=\beta\circ S$ and 
we have the relation:
\begin{equation}\eqlabel{squareantipDoH}
s^2(\varphi \#  h)=(\beta ^{-1}\# 1_H)
(\varphi \# h)(\beta\# 1_H),~\forall~\varphi\in H^*,~h\in H. 
\end{equation}

Let now $G$ be a finite group, with multiplication denoted by juxtaposition and unit denoted by $e$. 
Let $\omega$ be a normalized 3-cocycle on $G$, i.e.  
$\omega :G\times G\times G\rightarrow k^*$ is a map such that
$\omega (x, y, z)\omega (tx, y, z)^{-1}\omega (t, xy, z)\omega (t, x, yz)^{-1}
\omega (t, x, y)=1$ for all $t, x, y, z\in G$, and $\omega (x, y, z)=1$
whenever $x$, $y$ or $z$ is equal to $e$. We can take $H=k[G]$, the group algebra of $G$, which 
is a finite dimensional cocommutative Hopf algebra, 
and extend $\omega $ by linearity to a map $\omega :H\ot H\ot H\rightarrow k$, which 
turns out to be a Sweedler 3-cocycle on $H$. So, we can consider 
the QT quasi-Hopf algebra $D^{\omega }(H)$, which will be denoted by 
$D^{\omega }(G)$. This QT quasi-Hopf algebra structure of $D^{\omega }(G)$ was introduced in \cite{dpr}. 

\section{Some ribbon elements for $D^{\omega}(H)$ and $D^{\omega}(G)$}\selabel{ribbDrinDoubqHA}
\setcounter{equation}{0}
Let $H$ be a finite dimensional cocommutative Hopf algebra, $\omega$ a normalized $3$-cocycle on $H$ and 
$D^{\omega}(H)$ the quasi-Hopf algebra constructed in \seref{Domega}. So 
$D^{\omega}(H)$ is a QT quasi-Hopf algebra isomorphic to the quantum double $D(H^*_\omega)$. 

In what follows, in order to avoid any confusion, we denote by $\mu_H\in H^*$ and $\un{\mf g}_H\in H$ 
the modular elements of $H$ as a Hopf algebra. Similar notation, $\mu_{H^*_\omega}\in H$ and $\un{\mf g}_{H^*_\omega}\in H^*$, 
is used for the modular elements of the quasi-Hopf algebra $H^*_\omega$. As $H$ is cocommutative it follows that 
$\un{\mf g}_H=\mu_{H^*_\omega}=1_H$, i.e. $H^*$ and $H^*_\omega$ are unimodular as Hopf and respectively quasi-Hopf algebras. 
Also, it is clear that a left (and at the same time right) integral in the quasi-Hopf algebra $H^*_\omega$ 
is nothing but a left (and at the same time right) integral on the Hopf algebra $H$, i.e. an element $\l\in H^*$ 
obeying $\l(h)h=\l(h)1_H$, for all $h\in H$. Finally, $K$ will always denote a finite dimensional quasi-Hopf algebra and 
$H$ a finite dimensional cocommutative Hopf algebra.     

In this section we show that the element $\nu=u(\zeta\# 1_H)\beta_{D^{\omega}(H)}=u(\zeta\b\# 1_H)$ 
is a ribbon element for $D^{\omega }(H)$, provided that $\zeta: H\ra k$ is an algebra map such that $\zeta^2=\mu_H$; 
here, as before, $u$ is the element in \equref{elmu} corresponding to 
$D^{\omega}(H)$ and the other notation is as in \seref{Domega}. Note that when ${\rm dim}_kH\not =0$ 
in $k$ or $H$ is unimodular we can take $\zeta=\va$, and therefore $D^\omega(H)$ is ribbon with 
ribbon element $\nu=u(\b\# 1_H)$. This applies for instance to a finite dimensional Hopf group algebra $H=k[G]$, 
and so $D^\omega(G)$ is always a ribbon quasi-Hopf algebra. 

To this end, we start by describing the space of left cointegrals on $H^*_\omega$. By 
the comments made after the proof of \cite[Theorem 3.7]{bcint2} we have that a non-zero left cointegral on a quasi-Hopf 
algebra $K$ for which $\a, \b$ are invertible elements is a non-zero morphism $\l\in K^*$ satisfying 
$\l(t_2)t_1=\l(t)\b S^{-1}(\a)$, for any left integral $t$ in $K$ (in \cite{bcint2} the assumption $\b$ invertible 
is omitted and in place of $S^{-1}(\a)$ appears $\a$; we correct these facts now). 

\begin{lemma}
Let $t\in H$ be a non-zero left integral in $H$, i.e. $ht=\va(h)t$, for all $h\in H$. Then ${\mf t}:=\beta(S(t))t\in H$ 
is a non-zero left cointegral for $H^*_\omega$. 
\end{lemma}
\begin{proof}
For the quasi-Hopf algebra $H^*_\omega$ the elements $\a, \b$ are invertible, so 
a non-zero left cointegral on $H^*_\omega$ is a non-zero element ${\mf t}\in H$ 
satisfying $\l(h{\mf t})=\l({\mf t})\b(h)$, for all $h\in H$. 

It is well known that for any non-zero left integral $t$ in $H$ the map $H^*\ni h^*\mapsto h^*(t)t\in H$ is bijective 
(result valid for any finite-dimensional Hopf algebra, not necessarily cocommutative). 
Thus we find a unique element $h^*\in H^*$ such that $h^*(t)t={\mf t}$; in particular, $h^*(t)\not=0$. 
We have, for all $h\in H$, that  
$h^*(t)\l(ht)=h^*(t)\l(t)\b(h)$, which is equivalent to  
$h^*(S(h)ht)\l(ht)=\l(t)h^*(1_H)\b(h)$, which in turn is equivalent to  
$\l(t)h^*(S(h))=\l(t)h^*(1_H)\b(h)$.
 
As $\l(t)\not=0$, it follows that $h^*=h^*(1_H)\b\circ S$, which implies $\b(S(t))\not=0$. Therefore, by rescaling, we can assume 
without loss of generality that ${\mf t}=\b(S(t))t$, as stated. 

Conversely, ${\mf t}=\b(S(t))t$ is a left cointegral on $H^*_\omega$ since  
\[
\l(h{\mf t})=\b(S(t))\l(ht)=\b(S(ht)h)\l(ht)=\b(h)\l(ht)=\b(h)\l(t),
\]
and this is equal to $\l({\mf t})\b(h)$ because 
\[
\l({\mf t})\b(h)=\b(S(t))\l(t)\b(h)=\l(t)\b(S(1_H))\b(h)=\l(t)\b(h).
\]
So our proof ends.
\end{proof}

For a finite-dimensional quasi-Hopf algebra $K$ the modular element $\un{g}\in K$ is defined by 
$\un{g}=\l(S^{-1}(q^2t_2p^2))q^1t_1p^1$, where $\l\in K^*$ is a left cointegral, $t\in K$ is a left integral 
such that $\l(S^{-1}(t))=1$ and $p_R=p^1\ot p^2$ and $q_R=q^1\ot q^2$ are the elements defined in \equref{pqr}. 

\begin{corollary}\colabel{modelemquantdoubleDomega}
We have that $\un{\mf g}_{H^*_\omega}=\b^2\mu_H$. Consequently, the modular element 
$\un{\mf g}_{D(H^*_\omega)}$ of the quantum double $D(H^*_\omega)$ equals $1_H\Join \b^2\mu_H$. 
\end{corollary}
\begin{proof}
Recall that $\mu_H$ is defined by $th=\mu_H(h)t$, for all $h\in H$ and $t\in H$ a non-zero left integral. 
Also, since $H^*$ is unimodular we have $\l\circ S=\lambda$. 
Thus, by specializing the above definition of $\un{g}$ to $H^*_\omega$ we compute, for all $h\in H$, that  
\begin{eqnarray*}
\un{\mf g}_{H^*_\omega}&=&q^2(S({\mf t}))p^2(S({\mf t}))\l(S({\mf t}h))q^1(S(h))p^1(S(h))\\
&=&\omega(S(h), S({\mf t}), {\mf t})\b(S({\mf t}))\omega^{-1}(S(h), {\mf t}, S({\mf t}))\l(S({\mf t}h))\\ 
&=&\omega(S(h), hS({\mf t}h), {\mf t}hS(h))\b(hS({\mf t}h))\omega^{-1}(S(h), {\mf t}hS(h), hS({\mf t}h))\l({\mf t}h)\\ 
&=&\omega(S(h), h, S(h))\b(h)\omega^{-1}(S(h), h, S(h))\l({\mf t}h)\\
&=&\b(h)\b(S(t))\l(th)\\
&=&\b(h)\b(hS(th))\l(th)\\
&=&\b(h)^2\mu_H(h)\l(t).
\end{eqnarray*}
But the pair $(\l, {\mf t})$ obeys $\l({\mf t})=1$ or, equivalently, $\b(S(t))\l(t)=1$. The latter 
is equivalent to $\l(t)=1$, and so $\un{\mf g}_{H^*_\omega}=\b^2\mu_H$, as desired. 

Finally, we have $\mu_{H^*_\omega}=1_H$, hence the formula in \cite[Proposition 5.9]{bcint2} yields 
\[
\un{\mf g}_{D(H^*_\omega)}=1_H\Join \un{\mf g}^{-1}_{H^*_\omega}\circ S^{-3}=1_H\Join 
\un{\mf g}^{-1}_{H^*_\omega}\circ S=1_H\Join \un{\mf g}_{H^*_\omega},
\]
since $\un{\mf g}^{-1}_{H^*_\omega}=\b^{-2}\mu_H^{-1}$ with $\b^{-1}=\b\circ S$ and $\mu^{-1}_H=\mu_H\circ S$, 
and together with $S^2=\Id_H$ this implies $\un{\mf g}^{-1}_{H^*_\omega}\circ S=\un{\mf g}_{H^*_\omega}$. 
\end{proof} 

The computation performed before \cite[Proposition 3.3]{bp} ensures that 
the distinguished element $\b_{D^\omega(H)}=\b\# 1_H\in D^\omega(H)$ satisfies 
\begin{equation}\eqlabel{firststepribbonDomega}
\Delta(\b_{D^\omega(H)})=(\b_{D^\omega(H)}\ot \b_{D^\omega(H)})(s\ot s)({\mf f}_{21}^{-1}){\mf f},
\end{equation}
where ${\mf f}\in D^\omega(H)\ot D^\omega(H)$ is the Drinfeld twist. Consequently, the same relation is satisfied 
by any element of the form $\zeta\b\# 1_H\in D^\omega(H)$, provided that $\zeta: H\ra k$ is an algebra map. 

We have now all the necessary ingredients in order to prove the following:

\begin{theorem}\thlabel{someribbonDomega}
Let $H$ be a finite dimensional cocommutative Hopf algebra $H$, $\omega$ a normalized 
$3$-cocycle on $H$ and $\zeta: H\ra k$ an algebra map. Then the element 
$\nu=u(\zeta\beta\# 1_H)$ is ribbon if and only if $\zeta^2=\mu_H$.  
\end{theorem} 
\begin{proof}
The quasi-Hopf algebra $D^\omega(H)$ is unimodular; this follows from Theorem \ref{izodcpcp} and 
\cite[Theorem 6.5]{btfact}. Furthermore, by \coref{modelemquantdoubleDomega} and the definition of the isomorphism 
$w$ in (\ref{dcpcp}) we can see that the modular element $\un{\mf g}_{D^\omega(H)}$ of $D^\omega(H)$ is
\[
\un{\mf g}_{D^\omega(H)}= 
w(\un{\mf g}_{D(H^*_\omega)})=\un{\mf g}_{D(H^*_\omega)}\# 1_H=\b^2\mu_H\# 1_H.
\]
So, according to \coref{impribbquantumdouble}, it suffices to see when  
${\mf l}:=\zeta\b\# 1_H\in D^\omega(H)$ satisfies the relations
\begin{eqnarray}
&&{\mf l}^2=\b^2\mu_H\# 1_H,\eqlabel{ribbcheck1}\\
&&\Delta(l)=(l\ot l)(s\ot s)({\mf f}_{21}^{-1}){\mf f},\eqlabel{ribbcheck2}\\
&&{\mf l}S^2(\v\# h)=(\v\# h){\mf l}~,~\forall~\v\# h\in D^\omega(H).\eqlabel{ribbcheck3}
\end{eqnarray} 
It can be easily checked that \equref{ribbcheck1} is equivalent to $\zeta^2\b^2=\b^2\mu_H$, and 
the latter is equivalent to $\zeta^2=\mu_H$, because $H^*$ is commutative and $\b$ is convolution 
invertible. Also, the comments made after \equref{firststepribbonDomega} guarantees that \equref{ribbcheck2} is always satisfied. 

We look at \equref{ribbcheck3}. By \equref{squareantipDoH} we have that  
\equref{ribbcheck3} is equivalent to $(\zeta\# 1_H)(\v\# h)=(\v\# h)(\zeta\# 1_H)$, for all 
$\v\# h\in D^\omega(H)$. The last equation becomes  
$
\zeta\v\# h=\v(h\rh \zeta\lh S(h))\# h
$, 
and it holds for any $\v\# h\in D^\omega(H)$ since $\zeta$ is an algebra map and 
$H^*$ is commutative.
\end{proof}

We end this section with some concrete examples. In what follows by 
$G(H^*):=\{\zeta: H\ra k\mid \zeta~~\mbox{is an algebra map}\}$ 
we denote the set of grouplike elements of $H^*$, assuming, as before, that 
$H$ is a finite-dimensional cocommutative Hopf algebra. $G(H^*)$ is a group under convolution, 
and so the group Hopf algebra $k[G(H^*)]$ is a Hopf subalgebra of $H^*$. 
By the freeness theorem proved in \cite{NichZoeller} 
it follows that $\mid G(H^*)\mid$ divides ${\rm dim}_k(H)$. 

\begin{example}\exlabel{oddcaseribbon}
If $\mu_H$ has odd order in $G(H^*)$ then $D^\omega(H)$ is a ribbon quasi-Hopf algebra.
\end{example}
\begin{proof}
If $\mu_H$ has order $2m+1$ in $G(H^*)$ then $\zeta:=\mu_H^{m+1}: H\ra k$ is an algebra map such that 
$\zeta^2=\mu_H$. By \thref{someribbonDomega} we obtain that $u(\mu_H^{m+1}\b\# 1_H)$ is a ribbon element 
for $D^\omega(H)$. 
\end{proof}

\begin{example}
Suppose that either ${\rm dim}_k(H)$ or $\mid G(H^*)\mid$ is an odd number. Then 
$D^\omega(H)$ is a ribbon quasi-Hopf algebra. 
\end{example}
\begin{proof}
Since $\mid G(H^*)\mid$ divides ${\rm dim}_k(H)$, in either case we get that $\mu_H$ has odd order, and so 
\exref{oddcaseribbon} applies.
\end{proof}

The next example shows that the condition in \cite[Proposition 3.3]{bp} is equivalent to the unimodularity of $H$. 

\begin{example}\exlabel{canribbelemDoH}
The element $\nu=u(\b\# 1_H)$ is a ribbon element for $D^\omega(H)$ if and only if 
$H$ is unimodular.
\end{example}
\begin{proof}
Take $\zeta=\va$ in \thref{someribbonDomega}; we obtain that $\nu=u(\b\# 1_H)$ is a ribbon element if and 
only if $\mu_H=\va$, i.e. $H$ is unimodular.
\end{proof} 

The next example refers precisely to the ribbon structure on $D^\omega(G)$ defined by (5.18) in \cite{ac}.

\begin{example}
The quasi-Hopf algebra $D^\omega(G)$ is ribbon.
\end{example} 
\begin{proof}
We have $D^\omega(G)=D^\omega(k[G])$ and $k[G]$ is unimodular. Indeed, 
$t=\sum_{g\in G}g$ is a left and right integral in $k[G]$. 
\end{proof}

\begin{example}
Suppose that ${\rm dim}_k(H)\not=0$ in $k$ (this happens for instance when $k$ is of 
characteristic zero). Then $\nu=u(\b\# 1_H)$ is a ribbon element for $D^\omega(H)$. 
\end{example}
\begin{proof}
Since $S^2=\Id_H$, by the trace formula proved in \cite[Proposition 2 (c)]{radtrace} we get that 
$H$ is semisimple, and so unimodular, too. Hence $\nu=u(\b\# 1_H)$ is a ribbon element 
for $D^\omega(H)$. 
\end{proof}


\end{document}